\documentclass[a4paper,11pt]{amsart}
\usepackage{amsmath,amsfonts,amssymb,amscd,amsthm,amsbsy,epsf}
\usepackage{graphicx,psfrag,comment}
\usepackage{amsrefs}
\numberwithin{equation}{section}
\newtheorem{theorem}{Theorem}[section]
\newtheorem{lemma}[theorem]{Lemma}
\newtheorem{cor}[theorem]{Corollary}
\newtheorem{proposition}[theorem]{Proposition}

\theoremstyle{remark}
\newtheorem{remark}[theorem]{Remark}
\newtheorem{remarks}[theorem]{Remarks}

\theoremstyle{definition}
\newtheorem{defn}[theorem]{Definition}

\def\B{{\mathcal B}}

\def\real{{\mathbb R}}
\def\complex{{\mathbb C}}
\def\torus{{\mathbb T}}
\def\zed{{\mathbb Z}}
\def\t{{\theta}}
\def\jt{{J^{-1}(K(\theta))}}
\def\jto{{J^{-1}(K(\theta+\omega))}}

\def\C{{\mathcal C}}

\newcommand{\X}{\ensuremath{\mathfrak{X}}} 
\renewcommand{\d}{\mathrm d}               

\DeclareMathOperator{\Id}{Id}
\DeclareMathOperator{\Ker}{Ker}         
\DeclareMathOperator{\avg}{avg}        
\begin{document}

\title[KAM for presymplectic maps]
{Tracing KAM tori in presymplectic dynamical systems}
\author{Hassan Najafi Alishah}
\thanks{Hassan Najafi Alishah, Mathematics department, Instituto Superior Tecnico,Av. Rovisco Pais
1049-001, Lisbon, Portugal, halishah@math.ist.utl.pt}
\author{Rafael de la Llave} 
\thanks{Rafael de la Llave, School of Mathematics, Georgia Institute of technology, 686 Cherry st. , Atlanta GA 30332, rafael.delallave@math.gatech.edu}


\begin{abstract}
We present a KAM theorem for presymplectic dynamical systems. The theorem has a `` a posteriori " format. We show that given a Diophantine frequency $\omega$ and  a family of presymplectic mappings, if we find an embedded torus which is approximately invariant with rotation $\omega$ such that the torus and the family of mappings satisfy some explicit non-degeneracy condition, then we can find an embedded torus and a value of the parameter close to to the original ones so that the torus is invariant under the map associated to the value of the parameter. Furthermore, we show that the dimension of the parameter space is reduced if we assume that the systems are exact.

\end{abstract}

\maketitle
\baselineskip=18pt              

\section{Introduction}
Presymplectic structures (constant rank, closed 2-forms) arise naturally in the study of degenerate Lagrangian and Hamiltonian mechanical systems with constrains, in time dependent Hamiltonian systems 
and  in control theory. (see, e.g., \cites{kunzle72a,kunzle72b,gn79,gn80,le89,muro92,delgado03}).
 
Given a presymplectic form $\Omega\in\Omega^2(M)$, a vector field $X\in\X(M)$ is said to be a Hamiltonian vector field associated with a function $H\in\C^\infty(M)$ if:
\[
 i_X\Omega=\d H .
\]
Due to the degeneracy of $\Omega$, there can be different functions $H$ associated with $X$, not differing by a constant. The corresponding flow $\phi^t_X:M\to M$ is a 1-parameter group of presymplectic diffeomorphisms: $(\phi^t_X)^*\Omega=\Omega$. Hence, the dynamics of such systems leave the presymplectic structure invariant. For other situations where presymplectic dynamics occur see, e.g., \cite{bhg99}.

Our aim is to state and prove a KAM type theorem for presymplectic dynamical systems, which extends the results of \cite{dgjv05} for the case of symplectic diffeomorphisms. 

Our main result can be stated as follows. Let $\torus^n=\real^n/\zed^n$ be the $n$-dimensional torus.
We consider the presymplectic manifold 
\begin{equation}
\label{manifold}
M:=T^*\torus^d\times\torus^n,
\end{equation}

with an \textbf{exact} presymplectic form $\Omega$ of rank $2d$, whose kernel coincides with the $\torus^n$-direction. See 
Remark~\ref{presymplectic1} and  
Remark~\ref{Hypothesistransversal} for more observations about the kernel. 

One says that $K:\torus^{d+n}\to M$ is an invariant torus of a diffeomorphism $f:M\to M$ with frequency $\omega\in\real^{n+d}$ if:
\[ f(K(\theta))-K(\theta+\omega)=0, \quad \forall \theta\in\torus^{n+d}. \]
When the left hand side is non-zero, but small enough (in 
some smooth norm that will be made explicit later), 
one says that $f$ has an \emph{approximate} invariant torus $K$ with frequency $\omega$ (this will be made precise later). Our main theorem can be stated in rough terms as follows:

\begin{theorem}
Let $f_\lambda:M\to M$, where $M$ is as  in \eqref{manifold}, be an analytic, non-degenerate in the sense to be defined later, $(2d+n)$-parametric family of presymplectic diffeomorphisms such that $f_0$ has an approximate invariant torus $K_0$, satisfying a non-degeneracy condition, with frequency $\omega$ satisfying a Diophantine condition. Then there exists a diffeomorphism $f_{\lambda_\infty}$ in this family, where $\lambda_\infty$ is close to $0$, which has an invariant torus $K_\infty$ with frequency $\omega$ and which is \textbf{``close"} to the initial torus $K_0$.
\end{theorem}

The precise version of the theorem will be stated below in Section \ref{sec:results}. We need to formulate precisely the non-degeneracy conditions 
and to make more precise the  definition of ``close" which requires introducing norms. 
Note that we will not require the system to be neither nearly integrable, nor to be written in action-angle variables. Indeed, just like in \cite{dgjv05} for the symplectic case, the fact that the dynamics of the system preserve the presymplectic structure implies that the KAM tori are automatically approximately reducible. This leads to an approximate solution of the linearized equations without transformation theory. Moreover, the reducing transformation is given explicitly in term of the approximately translated torus, which form the basis of an efficient numerical algorithm (an explicit description of this algorithm can be found in \cite{dela01}). We will discuss all this in more detail in the Section \ref{sec:results} below.

The proof of our main theorem follows an approach  similar to 
the one developed in \cite{dgjv05} for the symplectic case. The presymplectic case however has a few peculiarities due to the degeneracy of the 2-form.
For some of the most routine calculations, we will just refer to 
this paper.  The quasi-Newton method used here (and in \cite{dgjv05}) is of the type introduced by Moser in \cites{mo66a,mo66b}. We note that the approach is not based on 
transformation theory, which seems problematic in the case of 
presymplectic mappings since generating functions are not
as straightforward as in the symplectic case (see \cites{gollo84, caigo85,caigo87} for studies of the theory of canonical transformations) and the Lie transform method is hampered by the fact that there are several Hamiltonians that give the same vector field. The approach is based on 
deriving a parameterization equation and applying corrections additively. 
The presymplectic geometry leads to cancellations that reduce a Newton 
step to the constant coefficients cohomology equations customary in KAM theory. 
We also note that the same cancellations lead to very effective numerical 
algorithms.

We will also prove a flux-type vanishing lemma for exact presymplectic diffeomorphisms. Roughly speaking, we will show that the average of the translation is zero in the directions other than the ones tangent to torus in the basis. Note that in the  directions tangent to torus the averaging does not need to vanish. This also shows the need for considering a parametric family of diffeomorphisms, rather than just a single diffeomorphism.

There are two possible ways to extend the results, we have stated here for the maps, to flows. One is to use the local uniqueness we state in the theorem \ref{uniqtheo}  and the other one is to proceed with automatic reducibility for flows as in the symplectic case which is done in \cite{dgjv05}.  Both of them require some technical details  and it is work in progress to do so.

As a final note, it should be remarked that the results of this paper are not applicable to the dynamics associated with general Poisson structures. For regular Poisson structures, which have an underlying regular symplectic foliation, there are  cohomological obstructions to find a compatible presymplectic structure, see \cite{vaisman}. Even when these obstructions vanish (e.g., locally around invariant tori), so that one can find a compatible presymplectic structure, Poisson diffeomorphisms do not coincide with presymplectic diffeomorphisms, and these two kinds of diffeomorphisms have quite distinct properties. The KAM theory for Poisson manifolds has been developed in \cite{yi02}. In section \ref{sec:comparison}, we will compare the Poisson and presymplectic cases. 

\begin{remark} \label{presymplectic1} 
An important well known fact about  presymplectic forms $\Omega$  is that the kernel of 
$\Omega$ is an integrable distribution. 

We recall that the kernel of 
a form is 
\[
{\rm Ker}(\Omega) = \{Z  | i_Z(\Omega) = 0\} = \{ Z | \Omega(Z, X) = 0 \ \forall X \}
\]

Note that, for a general $2$-form $\Omega$ and any three vector 
fields $X,Y,Z$,  we have: 
\[
\begin{split} 
\d \Omega(X,Y,Z) & =  X(\Omega(Y,Z)) - Y(\Omega(X,Z)) + Z(\Omega(X,Y)) \\
& - \Omega( [X,Y] Z)  + \Omega( [X,Z], Y ] ) - \Omega(  [Y,Z], X) 
\end{split} 
\]

If $\d\Omega = 0$ and $X, Y$ are in the kernel of $\Omega$, 
for any $Z$, we have 
$\Omega( [ X, Y], Z) = 0$. 

Hence, if $X, Y \in {\rm Ker}(\Omega)$, $[X,Y] \in {\rm Ker}(\Omega)$. 
This shows that the distribution given by the kernel can be integrated to 
a manifold. 

Of course, in a torus, it could well happen that the leaves integrating 
the kernel are not compact (e.g. they could be an irrational foliation). 
\end{remark} 

\subsubsection{Some examples and comments} 
One example to keep in mind could be the three dimensional torus 
endowed with a presymplectic form $\Omega = d \Psi_1 \wedge d \Psi_2 $. 
Clearly, the kernel is given by the level sets of $\Psi_1,\Psi_2$. 

A more complicated example on $\torus^3$ is 
 $\Omega = d \Psi_1 \wedge \gamma$  where $\gamma$ is a closed but 
not exact form. In this case, the kernel can be an irrational foliation. 

Another example related to the previous ones is 
the study of quasi-periodically perturbed 
Hamiltonian systems $H(x, \omega t)$. These can be made autonomous 
by adding an extra variable $\theta \in \torus^d $ that satisfies 
$\frac{d}{dt}\theta = \omega$.  The phase space is now supplemented 
by a factor $\torus^d$. The symplectic form  in the phase 
space becomes  a presymplectic form in the extended 
phase space having $\torus^d$  in the kernel. Even this elementary 
example was considered as  covered by the KAM theory of symplectic systems
at the time of writing \cites{qpgeodesic}. 

The theory of presymplectic manifolds was developed (e.g. in \cite{gnh78})
to give a geometric framework to the Dirac theory of constrained systems,
\cite{dirac1,dirac}. 
There are many physically interesting examples of 
constrained systems to which the present theory 
applies. Notably, besides the examples in 
\cite{dirac1,dirac}, 
the papers \cites{kunzle72a,kunzle72b} contain a very concrete example of 
a relativistic system of spinning particles which is close 
to integrable.

The paper \cite{delgado03} shows how the Pontriaguin maximum principle 
for optimal trajectories can be formulated using presymplectic systems. 
If we consider a mechanical system with KAM tori and subject it to 
a control indexed by enough parameters, the results in this paper 
give a condition which ensures that the one adjust parameters 
to maintain the quasi-periodic motion. It would be interesting to 
study in detail concrete models, specially because the methods of 
this paper are well suited for numerical implementations. 

Note that, in contrast with symplectic manifolds, 
presymplectic manifolds may be odd dimensional. 
Hence, it is clear that an extension of the symplectic theory to presymplectic systems
will require significant modifications. A general theory of 
perturbations of quasi-periodic motions independent of 
geometric structures was undertaken in \cites{mo67}. Note, however 
that the perturbations by presymplectic systems do not satisfy the 
assumptions of \cites{mo67} because the normal eigenvalues of 
a torus do not change when we change the parameters. 
Of course, in the presymplectic case the origin of this degeneracy 
is the preservation of the geometric structure, which also helps 
by eliminating some of the perturbing terms. This is the interpretation of
the geometric identities used in this paper: they kill several 
dangerous perturbation terms. 

\section{Preliminaries and Motivation}

In this section we will fix some notations and state a few preliminary results. Along the way, we will also justify the assumptions that will appear later in our main result. 

As stated in the introduction, we consider $M=\torus^d\times\real^d\times\torus^n$ equipped with a constant rank exact presymplectic structure, i.e., an exact 2-form $\Omega\in\Omega^2(M)$, such that its kernel is:
\[ N:=\Ker\Omega=\{(u,(0,0,z))\in TM~|~u\in M, z\in\real^n\}. \]

The exactness assumption places restrictions on the presymplectic form but for applications to Hamiltonian dynamical systems with constrains or degenerate Lagrangian systems this is not too restrictive. For these systems the phase space is often obtained by restriction to a submanifold where the 2-form is the pullback of the canonical symplectic structure on the cotangent bundle (see \cite{gnh78}) or some other exact symplectic form (see \cite{gn79}). 

Let $V=\{(u,(x,y,0))\in TM|u\in M, (x,y)\in\real^{2d}\}$ so that $TM=V\oplus N$, and denote by $\pi : TM\rightarrow V $ the canonical projection on $V$. For each $z\in TM$, we have the linear isomorphism  $\tilde{J}(u): T_zM\rightarrow T_zM$ defined by:
\begin{equation} 
\Omega _u(\xi,\eta) = \langle \xi,\tilde{J} (u)\eta \rangle, \quad \xi, \eta\in T_uM 
\label{2.0} 
\end{equation}
where 
\[ \tilde{J}(u)=\left [ \begin{array}{cc}
J(u)&0\\
0&0
\end{array} \right ] \]
and $\langle \cdot,\cdot \rangle $ denotes the standard Euclidean inner product on $\real^{2d+n}$. The skew-symmetry of $\Omega$ implies that $J^\intercal =-J.$

We will be using the following norms. If $x=(x_1,...,x_{d+n})\in\real^{d+n}$ we set:
\[ |x|:= \max_{j=1,..,d+n}|x_j|. \]
For an analytic function $g$ on a complex domain $\mathcal{B}$ we denote by $|g|_{\mathbb{C}^m,\mathcal{B}}$ its $\mathbb{C}^m$-norm:
\[ |g|_{\mathbb{C}^m,\mathcal{B}}:=\sup_{0\leq|k|
_{\mathbb{Z}}\leq m}\sup_{z\in\mathcal{B}}|D^kg(z)|,\]
where $|l|_{\mathbb{Z}}:=|l_1|+... + |l_{d+n}|$.

We will be looking for real analytic invariant tori which extend holomorphically to a small strip in the complex space. More precisely,  let $U_\rho$ denote the complex strip of width $\rho>0$: 
\[ U_\rho=\{\theta\in\mathbb{C}^{d+n}/\zed^{d+n}: |\text{Im}(\rho)|\leq\rho\},\]
and introduce the following family of maps.
\begin{defn}
\label{toriset}
The space $({\mathcal{P}}_\rho, \|.\|_\rho)$ consists of functions $K:U_\rho\rightarrow M$ which are one periodic in all their arguments, real analytic on the interior of $U_\rho$ and continuous on the closure of $U_\rho$. We endow this space with the norm
\begin{equation}
 \|K\|_\rho := \sup_{\theta\in U_\rho} |K(\theta)|,
 \label{ronorm}
 \end{equation}
\end{defn}
which makes it into a Banach space.

We will also use the same notations for functions taking values  in 
vector spaces or in matrices.

Some well known results about the spaces above are 
the Cauchy bounds below (a consequence of Cauchy's integral representation 
of the derivative). For $0 < \delta < \rho$, we have:
\begin{equation} \label{Cauchybounds} 
\|  D^j K \|_{\rho - \delta} \le C_{j} \delta^{-j} \| K \|_\rho
\end{equation}

Like in all other KAM type results we will have to deal with small divisors. For that we set:

\begin{defn}\label{dioph}
Given $\gamma > 0$ and $\sigma\geq d+n $, we will denote by $D(\gamma,\sigma)$ the set of frequency vectors $\omega\in\real^{d+n}$ satisfying the \textbf{Diophantine condition}:
\begin{equation} 
|l \cdot \omega-m|\geq \gamma|l|_{\mathbb{Z}}^{-\sigma} \quad \forall l\in\mathbb{Z}^{d+n}  \backslash  \{0\}
, m\in\mathbb{Z}
\end{equation}
\end{defn}

The aim of this paper is to find invariant tori of a given frequency $\omega$ for a m-parametric family of  presymplectic  diffeomorphism $f_{\lambda}$, defined as follows:

\begin{defn}\label{mpara}
A \textbf{$m$-parametric family of presymplectic diffeomorphisms} $f_{\lambda}$ is a function
\[f:M\times B\rightarrow M, \qquad B\subseteq\real^m,\] 
such that for each $x\in M$ the map $f(x,\cdot)$ is of class $C^2$ and for each $\lambda\in B $ the map $f_{\lambda} := f(\cdot,\lambda)$ is a real analytic presymplectic diffeomorphism.
\end{defn}

We will introduce an algorithm to solve the equation
\begin{equation}\label{noerror}
  f_{\lambda}(K(\t))-K(\t+\omega)=0, \quad \omega\in D(\gamma,\sigma),
\end{equation}
given that one knows an approximate solution $K_0(\t)$ for the diffeomorphism $f_{\lambda_0}$, where, with out loss of generality, we will set $\lambda_0=0$. In other word, we know that
\begin{equation}\label{error}
 f_{\lambda_0}(K_0(\t))-K_0(\t+\omega)=e_0(\t),
\end{equation}
where the error term $e_0(\t)$ has small enough norm. Equation \eqref{noerror} will be solved by Newton method where at each step we have infinitesimal equations given by
\begin{equation}
\label{linearequ}
Df_{\lambda_i}(K_i(\t)) \Delta_{i}(\t)-\Delta_{i}(\t+\omega)+\left.\frac{\partial f_\lambda(K_i(\t))}{\partial\lambda}\right|_{\lambda=\lambda_i}\varepsilon_i= - e_i(\t).
\end{equation}
The approximate invariant tori and the geometry of the problem will lead us to a change of variables that will reduce \eqref{linearequ} to a simpler equation with constant coefficients that can be solved by the following result of R\"ussmann which will also be useful in some other proofs.

\begin{proposition}[\cites{Rus75,dela01}]
\label{difference}
Let $\omega\in D(\sigma,\gamma)$ and assume that $h:\torus^{d+n}\rightarrow \real^{2d+n}$ is analytic on $U_{\rho}$ and has zero average, $\avg(h)=0$. Then for all $0<\delta<\rho$ , the difference equation 
\begin{equation}\label{diffequ}
v(\t)-v(\t+\omega)=h(\t)
\end{equation}
has a unique zero average solution $v:\torus^{d+n}\rightarrow\real^{2d+n}$ which is analytic in $U_{\rho-\delta}$. Moreover, this solution satisfies the following estimate:
\begin{equation}\label{est1}
\|v\|_{\rho-\delta}\leq c_0\gamma^{-1} \delta^{-\sigma} \|h\|_{\rho},
\end{equation}
where $c_0$ is a constant depending on $n$ and $\sigma$.
\end{proposition}

\subsection{Lagrangian properties of invariant tori} A first, very important, consequence of the Diophantine condition on $\omega$ and the exactness of the presymplectic form is that invariant tori are actually Lagrangian tori:

\begin{lemma}\label{exactlag}
If $K(\t)\in\mathcal{P}_{\rho}$ is a solution of  \eqref{noerror} then $K^*\Omega$ is identically zero.
\end{lemma}
\begin{proof} Since $K(\t)$ satisfies \eqref{noerror} and $f$ is presymplectic we have
\[K^*\Omega=(K \circ T_{\omega})^*\Omega,\]
where $T_{\omega}(\t)=\t+\omega$. Moreover, since  $\omega$ is rationally independent, rotations on the torus are ergodic and this implies that $K^* \Omega$ is constant. If we write $K^* \Omega$ in matrix form, exactly as we did for $\Omega$ in \eqref{2.0}  we have
\begin{equation}\label{L}
K^*\Omega(\xi,\eta)=\langle\xi,L(\t)\eta\rangle \quad \xi,\eta\in T_{\t}(\torus^{d+n})
\end{equation}
where $L(\theta)$ is actually constant. It remains to show that $L(\t)\equiv 0$.

The $2$-form $\Omega$ is exact, so we can write $\Omega=\d\alpha$ where 
\[\alpha(u)=a(u) \d u,\quad a(u)=(a_1(u),...,a_{2d+n}(u))^{\intercal}.\]
Then we find that
\[(K^*\alpha)=\sum_{j=1}^{d+n}C_j(\t)\d\t^j\]
where the components $C_j$ have the following expression 
\[C_j(\t)=DK(\t)a(K(\t))_j.\]
This implies $L(\t)=DC(\t)^{\intercal}-DC(\t)$. But now:
\begin{equation}\label{avg}
\mbox{avg}(DC(\t)):= \int_{\torus^{d+n}}DC(\t)\d\t=0,
\end{equation}
which shows that:
\[\mbox{avg}(L(\t))=0.\]
But $L(\t)$ being constant, we conclude that $L(\t)=0$, i.e., $K^*\Omega=0$.
\end{proof}

Following simple lemma  extends the result of the Lemma \ref{exactlag} to approximate invariant tori:

\begin{lemma}\label{approxlag}
Let $f_0:M\to M$ be a presymplectic analytic diffeomorphism and let $K\in\mathcal{P}_{\rho}$ be an approximate invariant  torus with frequency $\omega\in D(\gamma,\sigma)$:  
\begin{equation}\label{approxinvar}
f_0(K(\t))-K(\t+\omega)=e(\t).
\end{equation}
and assume that $f_0$ extends holomorphically to some complex neighborhood of the image of $U_\rho$ under $K$:
\[ {\mathcal{B}}_r = \{z\in \complex^{2d+n} : \sup_{\t\in U_\rho}|z-K(\theta)| < r\}. \]
Then  there exist a constant $C>0$, depending on $n,\sigma,\rho,\|DK\|_{\rho}, |f_0|_{C^1,{\mathcal{B}}_r}$ and $|J|_{C^1,\B_r}$, such that for $0<\delta<\frac{\rho}{2}$ 
\begin{equation}\label{approxlag1}
\|L\|_{\rho-2\sigma}\leq C\gamma^{-1}\delta^{-(\sigma+1)}\|e\|_{\rho}
\end{equation}
where $L$ is the matrix representing the pullback form $K^*\Omega$ (see \ref{L}).
\end{lemma}

\begin{proof}
Let $g:=L-L\circ T_\omega$. Then, 
we note that $g$ is the expression in 
coordinates of 
\[
K^* \Omega  - T_\omega^* K^* \Omega = 
K^* f^*_0 \Omega  - T_\omega^* K^* \Omega
\]
Hence, when $K$ is exactly invariant $g = 0$. 
One can also easily show that $|| g || \le || De ||$
See  \cite{dgjv05} for more details. 

Using Proposition \ref{difference}, one obtains that:
\[ 	\|L\|_{\rho-2\delta}\le c_0\gamma^{-1} \delta^{-\sigma} \|g\|_{\rho-\delta}.\]
One can bound the norm of $g$ in exactly the same way as in the symplectic case, which can be found in \cite{dgjv05}, to obtain the result.
\end{proof}

\subsection{Automatic reducibility near invariant tori}

In this subsection we will assume that $K(\t)$ is an invariant torus of $f$, i.e., a solution of \eqref{noerror}. When one starts instead with an approximate invariant torus $K_0(\t)$ of $f$, i.e., a solution of \eqref{error}, the results of this subsection do not hold anymore. However, we will see in the next sections that we have versions of these results which hold in the approximate case and which will allow us to perform the Newton method and conclude the existence of an invariant torus.
For $K(\theta)\in{\mathcal{P}}_{\rho}$ let us decompose its Jacobian in the form
\begin{equation}
\label{xz}
DK(\theta)=(X(\theta),Z(\theta))
 \end{equation}
where $X(\theta), Z(\theta)$ are the first $d$ and last $n$ columns of $DK(\theta)$.  Also, for every vector in  $TM=V\oplus N$, we will use the subscripts $V$ and $N$ for the first and second projections in each factor. Assume that $K(\t)$ solves \eqref{noerror} and that there exists a $d\times d$-matrix valued function $N(\t)$ such that 
\begin{equation}
\label{nondegi}
 N(\t)(X_V^\intercal(\t)\cdot X_V(\t)) =\Id,
\end{equation}
where $X(\t)$ is as in \eqref{xz}  This non-degeneracy assumption will turn out to be one of the ingredients  to solve \eqref{noerror} approximately. Also, set\footnote{We will often abuse notation and will use $Y(\t)$ to denote both $Y_V(\t)$ and $Y(\t)$.}:
 \begin{equation}\label{y}
Y_V(\theta):=X_V(\theta) N(\theta) \quad \mbox{and} \quad
Y(\t):=\left [ \begin{array}{c} 
Y_V(\t)\\0
\end{array} \right ].
\end{equation}
Then the following matrix will provide us the change of variable needed to reduce the linearized equations \eqref{linearequ} to a simple form:
\begin{equation}\label{M}
M(\theta):=\left (\begin{array}{ccc}X_V(\t)&\jt Y(\t)&Z_V(\t)\\
X_N(\t)&0&Z_N(\t)\end{array}\right ),
\end{equation} 
where $X,Z$ and $Y $ are defined in \eqref{xz} and \eqref{y}  respectively. The non-degeneracy assumption \eqref{nondegi}, together with the fact that $K(\t)$ is Lagrangian (Lemma \ref{exactlag}), show that:
 \begin{equation}\label{geom1}
 \Omega_{K(\theta)}(X(\theta), \jt Y(\theta))=I_d
 \end{equation}
 \begin{equation}\label{geom2}
 \Omega_{K(\t)}(X(\t),X(\t))=0
 \end{equation}
 \begin{equation}\label{geom3}
 \Omega_{K(\t)}(X(\t),Z(\t))=0
 \end{equation}
Therefore, $X(\theta)$, $\jt Y(\theta)$ and $Z(\theta)$ do not form a presymplectic basis along the torus $K(\theta)$, because neither $\Omega_{K(\t)}(\jt Y(\t),\jt Y(\t))$ nor  $\Omega_{K(\t)}(\jt Y(\t),Z(\t))$ have to be zero, but they do provide a basis where $\Omega$ takes a rather simple form. Moreover, as the following lemma shows, they transform the linearized equations \eqref{linearequ} into a simpler form:

\begin{lemma}
\label{transf:linear}
The set $\{X(\theta),\jt Y(\theta),Z(\theta)\}$ is a basis provided the matrix
 \begin{equation}\label{V}
V(\t)=\left [ \begin{array}{ccc}
0&I_d&0\\
-I_d& -Y^{\intercal}(\t)\jt Y(\t) & (\jt Y(\t))^\intercal J(K(\t))Z_V(\t)\\
X_N(\t) & 0 & Z_N(\t)
\end{array}
\right ]
\end{equation}
is invertible. In this case, we have:
\begin{align}
\label{comuti}Df(K(\t)).(X(\t),Z(\t))=&(X(\t+\omega),Z(\t+\omega)),\\
Df(K(\t))\jt Y(\t)=&X(\t+\omega)S_1(\t)+ \jto Y(\t+\omega)\Id  + \notag\\
\label{comutii}& \qquad\qquad\qquad\qquad+Z(\t+\omega)A(\t),
\end{align}
where $A(\t)$ and $S_1(\t)$ are matrices satisfying: \footnote{We emphasize that identity \eqref{comut1} holds  only when we have an invariant torus.
In Corollary~\ref{inver3}, we will prove that for approximately invariant tori
\eqref{comut1} holds up to an error which can be bounded 
by the error in the invariance equation.}

\begin{equation}\label{comut1}
 Df(K(\t))\cdot M(\t))= M(\t+\omega)\cdot \left [
 \begin{array}{ccc}
 I_d&S_1(\t)&0\\
 0&I_d&0\\
 0&A(\t)&I_n
 \end{array}\right ].
 \end{equation}
 
 \end{lemma}

\begin{proof}
Let 
\begin{equation}\label{Q}
 Q(\t):= \left[
\begin{matrix}
X_V^{\intercal}(\t)J(K(\t)) & 0\\
(\jt Y(\t))^\intercal J(K(\t)) & 0\\
0 & I_n
\end{matrix}  \right].
\end{equation}
The expression \eqref{M} for $M$ and relations \eqref{geom1}, \eqref{geom2} and \eqref{geom3}, give:
\begin{equation}
 \label{inver1}
 Q(\t)\cdot M(\t)=\left [
 \begin{array}{ccc}
 0&I_d& 0\\
 -I_d&-Y^\intercal(\t)\jt Y(\t)&(\jt Y(\t))^\intercal J(K(\t))Z_V(\t) \\
 X_N(\t)&0&Z_N(\t)
 \end{array} \right ],
 \end{equation}
which shows that $\{X_V(\t), \jt Y(\t)\}$ is a basis for $V:=\pi(TM)$ and $Q(\t)$ is invertible. Using this fact one can write
\[
Z_V(\t)=a^k_l(\t)X^l_V(\t)+b^k_l(\t) \jt Y^l,\qquad (l=1,\dots,d,\, k=1,\dots,n).
\]
Pairing both sides with $X_V^{l_0}(\t)$ via the presymplectic form $\Omega$, it follows from \eqref{geom1}, \eqref{geom2} and \eqref{geom3}  that:
\begin{equation}\label{b}
b^k_{l_0}(\t)=\Omega(X_V^{l_0}(\t),Z_V^k(\t))-a^k_l(\t)\Omega(X_V^{l_0}(\t),X_V^l(\t))=0.
\end{equation}
In general, we have no control on $\Omega(\jt Y(\t), Z_V(\t))$, it means we have no control on the $a^k_l(\t)$, but the assumption that $V(\t):= Q(\t)\cdot M(\t)$ is non-degenerate guarantees that $\{X(\theta),\jt Y(\theta),Z(\theta)\}$ is a basis.
 
Assume from now on that $V(\t)$, and hence $M(\t)$, is invertible. Since $f$ is presymplectic and $f(K(\t))=K(\t+\omega)$, is follows from \eqref{geom1}, \eqref{geom2} and \eqref{geom3} that:
\begin{align*}
Df(K(\t)).(X(\t),Z(\t))=&(X(\t+\omega),Z(\t+\omega)),\\
Df(K(\t))\jt Y(\t)=&X(\t+\omega)S_1(\t)+ \jto Y(\t+\omega)\Id  + \\
\label{comutii}& \qquad\qquad\qquad\qquad+Z(\t+\omega)A(\t),
\end{align*}
for some matrices $S_1(\t)$ and $A(\t)$. This shows that relations \eqref{comutii} hold.  Moving the term $\jto Y(\t+\omega)\Id$ to the left side of the second equation we obtain that:
 \begin{equation}
\label{A}
A(\t) = T_3(\t+\omega)\left [ Df(K(\t)) \jt Y(\t) - \jto Y(\t+\omega) \right ],
\end{equation}
where $T_3(\t)$ is the last row in the matrix:
\begin{equation}\label{M1}
  M^{-1}(\t)= \left [
\begin{array}{c}
 T_1(\t)\\
 T_2(\t)\\
 T_3(\t)
 \end{array}\right ].
 \end{equation}
Finally, moving the term $Z(\t+\omega)A(\t)$ to the left hand side and pairing both sides with $\jto Y(\t+\omega)$, via the presymplectic form $\Omega$,  together with \eqref{geom1}, gives:
\begin{align} \label{S1}
 S_1(\t) = \left[Y_V(\t+\omega))^\intercal\, 0 \right] &\left [ Df(K(\t)) \jt Y(\t)-\right.\\
 &\left. \jto Y(\t+\omega)-Z(\t+\omega)A(\t) \right ]. \notag
\end{align} 
\end{proof}
 
\begin{remark}\label{vinvers}
A straightforward calculation shows that $V^{-1}(\t)$ takes the following the form:
 \begin{equation}
 \left [ 
 \begin{array}{ccc}
 V_{11}^-&V_{12}^-&V_{13}^-\\
 I_d&0&0\\
 V_{31}^-&V_{32}^-&V_{33}^-
 \end{array} \right ].
 \end{equation}
 We will need this fact later.
 \end{remark}

\section{Main results}
\label{sec:results}
In this section we will give precise statements of our results. 
The discussion in the previous section motivates introducing the following:

 \begin{defn}\label{nondeg}
We will say that $K(\theta)\in\mathcal{P}_{\rho}$ is a \textbf{non-degenerate torus} if
\begin{enumerate}
\item[(i)] There exists a $d\times d$-matrix valued function $N(\theta)$ such that
\begin{equation}\label{nondeg1}
N(\theta)(X_V(\theta))^\intercal . (X_V(\theta))=\Id
\end{equation}
\item[(ii)] the matrix $V(\t)$, which defined in \eqref{V}, is invertible.
\end{enumerate}
where $X_V(\t)$ and $V(\t)$ are defined in \eqref{xz} and \eqref{V}  respectively.
 \end{defn}

\begin{remark} 
In the symplectic case, the matrix $V(\t)$ is always non-de\-ge\-ne\-ra\-te when $\Omega$ is exact. When $\Omega$ is not exact, then one also needs to assume that $V(\t)$ is invertible in order to perform the Newton iteration successfully. In the presymplectic case, even when $\Omega$ is exact, we need to assume that $V(\t)$ is invertible. Also, we do not know how to proceed with the algorithm presented here if one gives up on exactness of $\Omega$. However, one may still be able to proceed with this algorithm in some special problems where the form is non-exact. Dealing with KAM theory for non exact symplectic forms is a deep problem largely unexplored, see \cite{se08} for remarks on the 
problem of non-exact forms. 
\end{remark}
\begin{defn}\label{nondeg2}
A pair $(f_\lambda,K(\t))$ is \textbf{non-degenerate at $\lambda=\lambda_0$} if $f_\lambda$ is a $(2d+n)-$parameter family of presymplectic diffeomorphisms, $K(\t)\in\mathcal{P}_\rho$ is a non-degenerate torus, and the average of the $(2d+n)\times(2d+n)$ matrix
\begin{equation}\label{lambda}
\Lambda(\t):=V^{-1}(\t)Q(\t)\left(\left.\frac{\partial f_{\lambda}}{\partial\lambda}\right|_{\lambda=\lambda_0}
(K(\t))\right)
\end{equation}
has rank $2d+n$, where $V(\t)$ is defined by \eqref{V}. 
\end{defn}

\begin{remark}
As indicated in the outline above, the role of the non-degeneracy 
assumption is that we can transform the Newton equation 
to a constant coefficient equation (up to a small error). 

Note that the condition is an open condition, so that, if 
the initial error is small enough, the iterative process 
does not leave the region where Definition~\ref{nondeg2}
holds. 

\end{remark}
\begin{remark}
\label{Hypothesistransversal}
Note that the Definition~\ref{nondeg2}
involves the presymplectic form. 

In geometric terms, the assumption in 
the non-degeneracy condition  we need is that the tangent  space to 
the torus, its symplectically conjugated space and the space 
of tangents along the parameter in the form span the whole 
space. That is, that the direction of moving along the 
parameter space, compensates the kernel of the presymplectic form. 
We have formulated this assumption using the fact that the kernel corresponds
to some coordinates of the space due to the integrability of the kernel which comes for free from the closeness assumption of the presymplectic form.
\end{remark}

We can now state the main theorem of the this paper:

\begin{theorem}\label{main}
Let $\omega\in D(\gamma,\sigma)$,  let $f_{\lambda}$ be a $2d+n$-parametric family of analytic presymplectic diffeomorphisms and let  $K_0\in\mathcal{P}_{\rho_0}$. Assume that:
\begin{itemize}
\item[(H1)] The pair $(f_\lambda,K_0)$ is non-degenerate at $\lambda=\lambda_0$.
\item[(H2)] The family $f_\lambda$ can be holomorphically extended to some complex neighborhood of the image of $U_\rho$ under $K$:
\[ \mathcal{B}_r=\{z\in\mathbb{C}: \sup |z-K(\t)|<r\}\]
such that $|f_\lambda |_{C^2,\mathcal{B}_r}<\infty$.
\end{itemize}
If
\[ e_0(\t):=f_{\lambda_0}(K_0(\t))-K_0(\t+\omega),\]
then there exists constant $c>0$, depending on $\sigma$, $n$, $d$, $\rho_0$, $r$, $|f_{\lambda_0}|_{C^2,{\mathcal{B}}_r}$, $\|DK_0\|_{\rho_0}$, $\|N_0\|_{\rho_0}$, $\|\frac{\partial f_\lambda}{\partial\lambda}\mid_{\lambda=\lambda_0}(K_0)\|_{\rho_0}$ and $|\avg(\Lambda_0)^{-1}|$ such that if $0<\delta_0<\max(1,\frac{\rho_0}{12})$ and
\begin{equation}\label{cond}
\|e_0\|_{\rho_0} <\min\left\{ \gamma^{4}\delta_0^{4\sigma}, rc\gamma^{2}\delta_0^{2\sigma}\|e_0\|_{\rho_0}\right\}
\end{equation}
then there exists a mapping $K_\infty\in\mathcal{P}_{\rho_0-6\sigma_0}$ and a vector $\lambda_\infty\in\real^{2d+n}$ satisfying 
\begin{equation}\label{exact1}
f_{\lambda_\infty}\circ K_\infty =K_\infty\circ T_\omega
\end{equation}
Moreover, the following inequalities hold:
\begin{equation}\label{estfork}
\|K_\infty-K_0\|_{\rho_0-6\delta_0}<\frac{1}{c}\gamma^2\delta_0^{-2\sigma}\|e_0\|_{\rho_0}
\end{equation}
\begin{equation}\label{estforlamb}
|\lambda_\infty|<\frac{1}{c}\gamma^2\delta_0^{-2\sigma}\|e_0\|_{\rho_0}
\end{equation}
\end{theorem}

\begin{proof}[Sketch of the proof] 
More details of the proof will be given later, in 
Sections~\ref{sec:linearized:eq}~\ref{sec:improvedstep}~
\ref{sec:iteration}, but it will be useful to start with a
brief overview that can serve as a road map.

We will use a modified Newton method of the type introduced by Moser in \cites{mo66a,mo66b,ze75}. The procedure goes as follows. Starting with
\begin{equation}\label{G}
G(K_0,0):= f_0(K_0(\t))-K_0(\t+\omega)=e_0(\t),
\end{equation} 
we look for an approximate solution for the corresponding linearized equation
\begin{align}\label{DG}
\left.DG(K_0,0)\right|_{(\Delta_0(\t),\varepsilon_0)}:=&\\
\left.\frac{\partial f_\lambda(K_0(\t))}{\partial\lambda}\right|_{\lambda=0}\varepsilon_0&+Df_0(K_0(\t))\Delta_0(\t)-\Delta_0(\t+\omega)=-e_0(\t).\notag
\end{align}
The left hand side of this equation
By an approximate solution we mean up to a quadratic error, i.e., a solution $ \Delta_0(\t)$ such that:
\[  \|DG(K_0,0)|_{(\Delta_0(\t),\varepsilon_0)}+e_0\|_{\rho_0-\delta_0} \leq c_0\gamma^{-3}\delta_0^{-(3\sigma+1)}\|e_0\|^2_{\rho_0}\] 
where $\delta_0, c_0$ are constants to be determined later.

Having the solution $(\Delta_0(\t), \varepsilon_0)$ a better approximating torus for the map $f_{\lambda_1}$, where $\lambda_1=\lambda_0+\varepsilon_0$, is defined as 
\[K_1(\t)= K_0(\t)+\Delta_0(\t)\]
and it will be shown that $(K_1(\t), f_{\lambda_1})$ is a non-degenerate pair. Furthermore, setting
\[e_1(\t):= f_{\lambda_1}(K_1(\t))-K_1(\t)\]
we find that  
\[ \|e_1 \|_{\rho_0-\delta_0} \leq c_0 \gamma^{-4}\delta_0^{-4\sigma} \|e_0\|^2_{\rho_0}.\]
In other words, for the new torus the error has decreased quadratically. 

Iterating this procedure, we will see that the sequence 
\[ (K_0,\lambda_0), (K_1,\lambda_1),\dots,(K_n,\lambda_n),\dots \] 
of  approximate solutions of \eqref{noerror}, 
obtained by applying the iterative procedure, converges
to a solution $(K_\infty,\lambda_\infty)$. One has to be careful with the domain $U_\rho$ which decreases in each iteration
(the reason is because we can bound the correction applied 
at one step only in a domain slightly smaller than the domain of 
the original function). 
This loss of domain can be arranged in a way that, in the limit, one does not end up with an empty domain. This choice of 
decreasing domains so that there is 
some domain that remains  is very standard in KAM theory
since the first papers \cite{kol54,mo66a,mo66b}. See \cite{ze75,dela01}
for a  pedagogical exposition.
\end{proof}

\subsection{Local uniqueness}

Notice that if $K_\infty$ is a solution of \eqref{exact1} then for every $\varphi\in\torus^d\times\torus^n$ the map $K_\infty(\t+\varphi)$ is also a solution. For this reason, we will consider $K(\t)$ and $\hat{K}(\t):=K(\t+\varphi)$ to be equivalent. By \emph{uniqueness of solutions}, we will mean uniqueness up to this equivalence relation. 
The following result gives uniqueness of solutions of \eqref{exact1}:

\begin{theorem}\label{uniqtheo} 
Let $\omega\in D(\gamma,\sigma)$ and assume that $K_1$ and $K_2$ are two non-degenerate tori in $\mathcal{P}_\rho$ solving
\begin{equation}\label{uniq1}
f_\lambda(K(\t))-K(\t+\omega)=0,
\end{equation}
such that $K_1(U_\rho)\subset\mathcal{B}_r$ and $K_2(U_\rho)\subset\mathcal{B}_r$. Furthermore, assume that the matrix
\[\Theta:= \avg\left( \left[ \begin{array}{c} S_1(\t)\\A(\t) \end{array} \right] \right), \]  
where $S_1(\t), A(\t)$ are defined by \eqref{S1} and \eqref{A}  has rank $d$.  Then there exists a constant ${\tilde c} >0$ depending on $\sigma$, $n$, $d$, $\rho$, $r$, $|f_{\lambda}|_{C^2,{\mathcal{B}}_r}$, $\|DK_1\|_\rho$, $\|N_1\|_{\rho}$ and $|\Theta|$ such that if
\begin{equation}\label{conduniq}
\|K_1-K_2\|_{\rho}< {\tilde c} \gamma^2 \delta^{2\sigma},
\end{equation} 
where $\delta=\frac{\rho}{8}$, then there exists an initial phase $\tau\in\torus^d\times\torus^n$ such that in $U_{\rho/2}$ one has:
\[K_1\circ T_{\tau} =K_2\]
\end{theorem}
The proof of this result is given in section \ref{sec:uniqueness}

\subsection{A vanishing lemma}

We end this section with one geometric result. Recall that a diffeomorphism $f:M\to M$ is called \textbf{exact presymplectic} if at the level of de Rham cohomology one has:
\begin{equation}\label{exapres} 
[f^\ast\alpha-\alpha]=0
\end{equation}
where $\alpha$ is a primitive of the presymplectic form: $\Omega= \d\alpha$. When $M$ is not compact, one must use compactly supported de Rham cohomology. Clearly, the time-1 map of a Hamiltonian vector field is exact. Moreover, using the flux homomorphism (see \cite{ban97}), one can show that an exact presymplectic diffeomorphism which is close enough to the identity is the time-1 map of a (time-dependent) Hamiltonian vector field. 

We now generalize to exact presymplectic diffeomorphisms the Vanishing Lemma of \cite{fds09}, valid for exact symplectic diffeomorphisms, and which allows one to have some control on the size of the parameter $\lambda$. Due to the the presence of kernel, our Vanishing Lemma has a slightly different nature (and statement) than \cite{fds09}*{Lemma 4.9}.

We will assume that we are in the situation described in the statement of Theorem \ref{main}, where $f_0$ is exact. In order to simplify the notation we write $K(\t)$ instead of $K_\infty(\t)$ and $\lambda$ instead of $\lambda_\infty$. Let $\tilde{f_\lambda}:= f_\lambda-f_0$ and define the average\footnote{In the sequel, we will not distinguish between a map with values in $\torus^{d+n}\times\real^d$ and a lift with values in $\real^{2d+n}$.}
\begin{equation}\label{mu}
\bar{\mu}:=\int_{\torus^{d+n}} \tilde{f_\lambda}(K(\t))~\d\t \in\real^{2d+n}
\end{equation}
If we express the vector $\bar{\mu}$ in the basis $\{X(\t),J^{-1}(K(\t))Y(\t),Z(\t)\}$, we obtain the $\t$-dependent components $(\mu_1(\t),...,\mu_{2d+n}(\t))$, in other word
\[\bar{\mu}= \left[\mu_1(\t),...,\mu_{2d+n}(\t)\right]\left[\begin{array}{ccc}X(\t)& J^{-1}(K(\t))Y(\t) &Z(\t)\end{array}\right] .\]
We have
 \begin{lemma}[Vanishing Lemma]\label{vanlemma}
 If $f_0:M\rightarrow M$ is an exact presymplectic diffeomorphism, then 
 \begin{equation}\label{vanish1}
 \int_{\torus^{d+n}}\mu_{k}(\t)~\d\t=0,\qquad (k=d+1,\dots,2d).
 \end{equation}
 \end{lemma} 

\begin{proof} 
We fix the following notations
\[
\hat{\t}_i=(\t_1,...,\t_{i-1},\t_{i+1},...,\t_{d+n})\in\torus^{d+n-1}
\]
\[
\hat{\omega}_i=(\omega_1,...,\omega_{i-1},\omega_{i+1},...,\omega_{d+n})\in\real^{d+n-1}
\]
and we let $\sigma_{i,{\t}_i}:\torus\rightarrow\torus^{d+n}$ be the path given by;
 \[\sigma_{i,{\t}_i}(\eta) =(\t_1,...,\t_{i-1},\eta,\t_{i+1},...,\t_{d+n}).\]
Also, we consider the two-cell $B_{i,\hat{\t}_i}: [0,1]\times \mathbf{S}^1\to\real^{2d+n}$ defined by:
\begin{equation}\label{twocell} 
  B_{i,\hat{\t}_i}(\xi,\eta):=  K\circ \sigma_{i,\hat{\t}_i+\hat{\omega}_i}(\eta)-(\bar{\mu} )\circ\sigma_{i,\hat{\t}_i+\hat{\omega}_i}(\eta)\xi.
\end{equation}
We will compute the integral 
\[
\int_{B_{i,\hat{\t}_i}}\Omega
\]
in two distinct ways:

(1) The boundary of $B_{i,\hat{\t}_i}$ is the difference between the two paths $K\circ\sigma_{i.\hat{\t}_i+\hat{\omega}_i}$ and $(K\circ T_\omega-\bar{\mu})\circ\sigma_{i,\hat{\t}_i}$, so by Stokes's theorem we conclude
\begin{equation}\label{integ3}
   \int_{B_{i,\hat{\t}_i}}\Omega=\int_{(K\circ T_\omega-\bar{\mu} )\circ\sigma_{i,\hat{\t}_i}}\alpha-\int_{K\circ \sigma_{i,\hat{\t}_i+\hat{\omega}_i}}\alpha
\end{equation}
Since $(\tilde{f_\lambda}(K(\t))-\bar{\mu})$ has average zero and satisfies all hypothesizes of Proposition \ref{difference}, there exists an analytic function $v:\torus^{d+n}\to\real^{2d+n}$ such that
\[v(\t)-v(\t+\omega)=\tilde{f_\lambda}\circ K-\bar{\mu}.\]
This, together with the exactness of $f_0$ implies that: 
\begin{align*}
\int_{(K\circ T_\omega-\bar{\mu} )\circ\sigma_{i,\hat{\t}_i}}\alpha
&=\int_{(f_\lambda\circ K-\bar{\mu} )\circ\sigma_{i,\hat{\t}_i}}\alpha=\int_{(f_0\circ K+\tilde{f_\lambda}\circ K-\bar{\mu} )\circ\sigma_{i,\hat{\t}_i}}\alpha\\
&=\int_{K\circ \sigma_{i,\hat{\t}_i}}f_0^{\star}\alpha+\int_{v\circ\sigma_{i,\hat{\t}_i}-v\circ\sigma_{i,\hat{\t}_i+\hat{\omega}_i}}\alpha\\
&=\int_{K\circ \sigma_{i,\hat{\t}_i}}\alpha+\int_{v\circ\sigma_{i,\hat{\t}_i}-v\circ\sigma_{i,\hat{\t}_i+\hat{\omega}_i}}\alpha.
  \end{align*}
Hence, we see that:
\[
  \int_{B_{i,\hat{\t}_i}}\Omega=\int_{(K+v)\circ\sigma_{i,\hat{\t}_i}}\alpha-\int_{(K+v)\circ\sigma_{i,\hat{\t}_i+\hat{\omega}_i}}\alpha.
\]
By a simple change of variable, we see that if we integrate over the torus $\torus^{d+n-1}$ the right-hand side of the previous equation vanishes, so we can conclude that
\begin{equation}\label{integ4}
 \int_{\torus^{d+n-1}}\d\hat{\t}_i \int_{B_{i,\hat{\t}_i}}\Omega=0.
\end{equation}

(2) Next we compute the integral of $\Omega$ over $B_{i,\hat{\t}_i}$ explicitly as follows:
 \[
 \int_{B_{i,\hat{\t}_i}}\Omega=\int_0^1\int_0^1 \Omega_{B_{i,\hat{\t}_i}(\xi,\eta)}
 (\partial_\xi B_{i,\hat{\t}_i}(\xi,\eta),\partial_\eta B_{i,\hat{\t}_i}(\xi,\eta))~\d\xi \d\eta
\]
Since $\bar{\mu}$ is a constant vector, by \eqref{twocell} and \eqref{xz} we have for $i=1,\dots,d$:
 \begin{align*}
 \partial_\eta B_{i,\hat{\t}_i}
 &=\partial_{\t_i} K\circ \sigma_{i,\hat{\t}_i+\hat{\omega}_i}=X_i\circ \sigma_{i,\hat{\t}_i+\hat{\omega}_i}\\
 \partial_\xi B_{i,\hat{\t}_i},
 &=-(\bar{\mu})\circ\sigma_{i,\hat{\t}_i+\hat{\omega}_i}.
 \end{align*}
 So from the partial presymplectic basis relations \eqref{geom1}, \eqref{geom2} and \eqref{geom3} we conclude that:
 \begin{equation}\label{integ7}
 \int_{B_{i,\hat{\t}}(\xi,\eta)}\Omega=\int_0^1\mu_{d+i}\circ\sigma_{i,\hat{\t}_i+\hat{\omega}_i}(\eta)~\d\eta, \qquad (i=1\dots,d).
 \end{equation}
 Now, \eqref{integ4} and \eqref{integ7}  together show that:
  \[\int_{\torus^{d+n-1}}\d\hat{\t}_i\int_0^1\mu_{d+i}\circ\sigma_{i,\hat{\t}_i+\hat{\omega}_i}(\eta)~\d\eta=0 \qquad (i=1\dots,d),\]
 and this yields the result.
 \end{proof}
 
\begin{remarks} The following remarks illustrate the relevance of the Vanishing Lemma:  

\begin{itemize}
\item The Vanishing Lemma concerns invariant tori. It can be extended to the approximate case, as it is done in \cite{fds09} for the symplectic case, and assuming that the whole family $f_\lambda$ is exact presymplectic, it leads to a bound on the parameter, which shows that in every step the value of the parameter decreases with the error term. This can be useful in numerical schemes for finding invariant tori.
\item In dimension 2, a volume preserving diffeomorphism of $\mathbf{S}^1\times\real$ is the same as (pre)symplectic diffeomorphism. In this case, as shown by the proof above, the integral \eqref{vanish1} is the oriented area between a circle and its image by the map, as shown in Figure \ref{fluxpic}. This clearly shows that the vanishing of \eqref{vanish1} is an obstruction for the existence of invariant tori (see also \cite{dela01}).
\begin{figure}[htb]
\begin{center}
\leavevmode
\includegraphics[height=6cm]{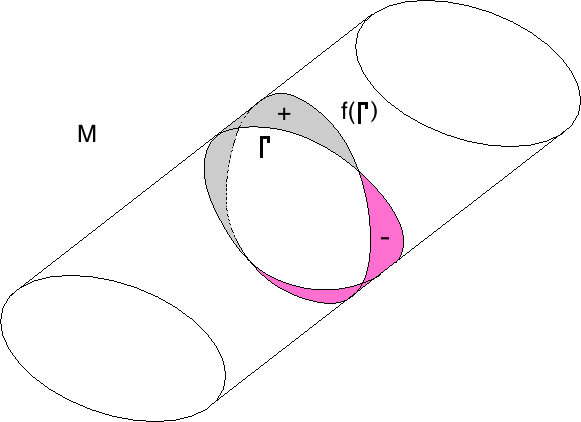}
\end{center}
\caption{Vanishing Lemma}
\label{fluxpic}
\end{figure}
\item Recall that we can think of our presymplectic manifold $M$ as $T^*\torus^d\times \torus^n$. In our Vanishing Lemma we only control the averages in the directions normal to $\torus^d$. It is easy to give simple examples of maps satisfying all the assumptions and such that the averages in other directions are non-zero. 
  \end{itemize}
\end{remarks}

\section{Estimates for the  linearized equation}
\label{sec:linearized:eq}
 
The sketch of the proof of Theorem \ref{main}, given in the previous section, relied on finding an approximate solution of the linearized equation \eqref{DG}, assuming that one has an approximate solution $K_0$ of \eqref{G}. In this section, we explained how this can be done.
 
 The first claim is that the set $\{X(\theta),\jt Y(\theta),Z(\theta)\}$ is still a basis for $T_{K_0(\t)}M$ if the error term is small enough.
 Note that now, due to the error term, equation \eqref{inver1} becomes 
 \begin{equation}\label{inver5}
 Q(\t)\cdot M(\t)=V(\t)+R(\t),
 \end{equation}
 where 
 \[ R(\t):=\left [
 \begin{array}{ccc}
 X_V^\intercal (\t) J(\t) X_V(\t)&0&X_V^\intercal (\t) J(\t) Z_V(\t)\\
 0&0&0\\0&0&0
 \end{array}\right ].
 \]
If we now use that $K_0(\t)$ is approximately Lagrangian, i.e., if we apply Lemma \ref{approxlag}, we see that we can control the reminder $R(\t)$:

 \begin{lemma}\label{VR}
   Assume the hypotheses of Lemma \ref{approxlag} hold. Then there exits a constant $c_3$ depending on $d$, $n$, $\rho$, $|f_\lambda|_{C^1,\mathcal{B}}$, $|J|_{C^1,\mathcal{B}}$, $\|N\|_\rho$, and $\|DK_0\|_\rho$ such that for every $0<\delta <\frac{\rho}{2}$ we have 
   \[
   \|V^{-1}\cdot R \|_{\rho-2\delta} \leq c_3 \gamma^{-1} \delta^{-(\sigma+1)}\|e_0\|_\rho.
   \] 
   \end{lemma}
   
 We conclude that:

 \begin{cor}\label{inver3}
  Assume the hypotheses of Lemma \ref{approxlag} hold. If $e_0(\t)$ satisfies
  \begin{equation}\label{coninver1}
  c_3\gamma^{-1}\delta^{-(\sigma+1)} \|e_0\|_\rho
\leq \frac{1}{2},
\end{equation}
 then $M$ is invertible and 
 \[ M^{-1}(\t)=V^{-1}(\t)Q(\t)+M_e(\t),\]
 where
 \begin{equation}\label{Me}
 M_e(\t)=- [ I_{2d+n}+V^{-1}(\t)R(\t)]^{-1} V^{-1}(\t)R(\t) V(\t)R(\t).
 \end{equation}
 Moreover 
 \begin{equation}\label{estMe}
 \|M_e\|_{\rho-2\delta} \leq c_4\gamma^{-1}\delta^{-(\sigma+1)}\|e_0\|_{\rho},
 \end{equation} 
 where $c_4$ is a constant which depends on the same parameters as $c_3$.
 \end{cor}
 
 \begin{proof} A simple application of the Neumman series. See \cite{dgjv05}.
 \end{proof}   
   
 We are ready to apply our change of variables. Before that we remark that, since $f_{\lambda_0}$ is presymplectic, we have 
\begin{equation}\label{df}
Df_{\lambda_0}(K(\t))=\left [ \begin{array}{cc}
F_1(\t)&0\\
F_2(\t)&F_4(\t)
\end{array}\right ],
\end{equation}
where $F_1(\t)$ is a symplectic linear map from $V=\pi(T_{K(\t)} M)$ into itself. 

 \begin{lemma}\label{change}
Let  $K_0(\t)\in\mathcal{P}_\rho$ solves 
 \[f_{\lambda_0}(K_0(\t))-K_0(\t+\omega)=e_0(\t)\]
 and that $(f_\lambda,K(\t))$ is non-degenerate at $\lambda=\lambda_0$ in 
the sense of definition \ref{nondeg2}. 
If $e_0(\t)$ satisfies \eqref{coninver1}, then the change of
variable $\Delta_0(\t)=M(\t)\xi(\t)$ transforms equation \eqref{DG} to
 \begin{align}
 \label{simplinear} 
 &\left ( \left [ 
 \begin{array}{ccc}
 I_d&S(\t)&0\\
 0&I_d&0\\
 0&A(\t)&I_n
 \end{array} \right ] + B(\t)\right ) \xi(\t)-\xi(\t+\omega)=\\
 &-V^{-1}(\t)Q(\t)e_0(\t)-\Lambda(\t)\varepsilon_0-M_e(\t)e_0(\t)-M_e(\t)(\left.\frac{\partial f_\lambda}{\partial\lambda}\right|_{\lambda=\lambda_0} )\varepsilon_0, \notag
 \end{align} 
 where 
 \begin{align*}
 B(\t)&:=M^{-1}(\t+\omega) E(\t)-\left [ \begin{array}{ccc}0&S_2(\t)&0\\0&0&0\\0&0&0\end{array}\right]\\
E(\t)&:=\left ( D_1e_0(\t), E_1(\t), D_2e_0(\t) \right ) \\
E_1(\t)&:=Df_{\lambda_0} (K_0(\t)) J^{-1}(K_0(\t)) Y(\t)-X(\t+\omega) S_1(\t) +\\
&\quad\qquad\qquad\qquad\qquad-J^{-1}(K_0(\t)) Y(\t+\omega)-Z(\t+\omega)A(\t)\\
 S_2(\t)&:=V^-_{13}\cdot \left ( F_2(\t)J^{-1}(K_0(\t)) Y(\t)-X_N(\t+\omega)S_1(\t)-Z_N(\t+\omega)A(\t)\right)\\
  S(\t)&:=S_1(\t)+S_2(\t),
 \end{align*}
 and $\Lambda(\t)$, $M_e(\t)$ and $S_1(\t)$ are defined by \eqref{lambda}, \eqref{Me} and \eqref{S1} respectively.
Moreover, we have the estimates:
 \begin{align}\label{estMe1}
 \|M_ee_0\|_{\rho-2\delta}&\leq c_4\gamma^{-1}\delta^{-(\sigma+1)}\|e_0\|^2_{\rho}\\
 \left\|M_e\left.\frac{\partial (f_\lambda \circ K_0)}{\partial\lambda}\right|_{\lambda=0}\varepsilon_0\right\|_{\rho-2\delta}
 &\leq c_4 \gamma^{-1}\delta^{-(\sigma+1)} \left\|\left.\frac{\partial (f_\lambda \circ K_0)}{ \partial\lambda}\right|_{\lambda=0}\right\|_\rho |\varepsilon_0| \|e_0\|_\rho \notag \\
 \label{estB}
 \|B \|_{\rho-2\delta}&\leq c_5 \gamma^{-1} \delta^{-(\sigma+1)} \|e_0\|_\rho
 \end{align}
 where $c_4$ is the same as in \eqref{estMe} and $c_5$ is another constant which depends on the same parameters . 
 \end{lemma}
 \begin{proof}
The form of the transformed equations follows from substituting the change of variable and elementary computations. 

To prove  the estimates \eqref{estMe1} and \eqref{estB}, 
we note that \eqref{estMe1} follows immediately from
 \eqref{estMe}, so it only  remains to prove \eqref{estB}.
 First note that for the first term in the definition of $B(\t)$ i.e. $M^{-1}(\t+\omega) E(\t)$, the Cauchy provide bounds for
 $D_1e_0(\t)$ and $D_2e_0(\t)$ in terms of the error. 
This enables us to  bound $M^{-1}(\t+\omega)\left( D_1e_0(\t),D_2e_0(\t)\right)$ by the error term. Calculating bounds for $M^{-1}(\t+\omega) E_1(\t)$ is more subtle. By the definition of $A(\t)$ and the fact that 
\[T_3(\t+\omega)X(\t+\omega)=0\] 
it follows that 
\[T_3(\t+\omega) E_1(\t)=0.\] 
Therefore:
 \[
 M^{-1}(\t+\omega)E_1(\t)=\left [ \begin{array}{c}
 \left [\begin{array}{c}T_1(\t+\omega)\\ T_2(\t+\omega)\end{array}\right ] E_1(\t)\\ 
 0\end{array}\right  ].
 \]
 By the corollary \ref{inver3} and the remark \ref{vinvers} we get:
  \[
  \left [ \begin{array}{c}
  T_1(\t) \\ T_2(\t) 
  \end{array}\right ] = \left [ \begin{array}{ccc}V_{11}^-&V_{12}^-&V_{13}^-\\Id&0&0\end{array}\right ]Q(\t)+\tilde{M_e}(\t),
  \]
   Where\footnote{$Q(\t)$ is defined at \eqref{Q}, just to make it easier to follow the calculations we restate it again} 
   \[ Q(\t):= \left[
\begin{matrix}
X_V^{\intercal}(\t)J(K(\t)) & 0\\
(\jt Y(\t))^\intercal J(K(\t)) & 0\\
0\hspace{1.2cm}0 & I_n
\end{matrix}  \right], \]
 and $\tilde{M_e}(\t)$ is obtained from $M_e(\t)$, defined at \eqref{Me}, by removing the last $n$ rows. So, we have
 \begin{align}
 \label{calcu2}
 \left [ \begin{array}{c}T_1(\t+\omega)\\ T_2(\t+\omega)\end{array}\right ]E_1(\t)=&
\overbrace{ \left [ \begin{array}{cc}
  \tilde{V}^{-1}(\t+\omega) \tilde{Q}(\t+\omega) & \begin{array}{c}0\\0\end{array}
  \end{array}\right ]E_1(\t)}^{(1)}+\\
  &\qquad+\underbrace{\left[ \begin{array}{ccc}0&0&V^-_{13}\\0&0&0\end{array} \right]E_1(\t)}_{(2)} + \underbrace{\tilde{M_e}(\t) E_1(\t)}_{(3)}, \notag
  \end{align}
  where we used notations:
\[\tilde{V}^{-1}(\t):= \left[ \begin{array}{cc}V_{11}^-&V_{12}^-\\I_d&0
  \end{array} \right], \]
\[ \tilde{Q}(\t)=\left [\begin{array}{c}X_V^{\intercal}(\t)J(K(\t))\\
(\jt Y(\t))^\intercal J(K(\t)) \end{array}\right].\]
Note that, by \eqref{estMe}  the term  $(3)$ in the right hand side of \eqref{calcu2} is bounded by the error  i.e.,
\[
\|\tilde{M_e}(\t) E_1(\t)\|_{\rho-2\delta}\leq c_6\gamma^{-1}\delta^{-(\sigma+1)}\|e_0\|_\rho,
\]
where $c_6$ depends on $c_4$ from \eqref{estMe} and $\|E_1(\t)\|_\rho$\footnote{As we will see $\|E_1(\t)\|_\rho$ contains terms that are not bounded by the error, so we do not get quadratic bound by the error as in the symplectic case and the constant depends on $\|E_1(\t)\|_\rho$ also.}.  Considering \eqref{df} and an elementary computation shows that
\begin{equation}\label{e1}
E_1(\t)=\left[ \begin{array}{c}
F_1 (\t) \jt Y(\t)-X_V(\t+\omega) S_1(\t) -\jto Y(\t+\omega)-Z_V(\t+\omega)A(\t)\\
F_2(\t)\jt Y(\t)- X_N(\t+\omega)S_1(\t)-Z_N(\t+\omega)A(\t)
\end{array}\right ],
\end{equation}
substituting \eqref{e1} in the term $(1)$ of  left hand side of \eqref{calcu2}, we get that  term  $(1)$ is equal to
\begin{equation}\label{calcu5}
\tilde{V}^{-1}(\t+\omega)\cdot \left [ \begin{array}{c}
X_V^\intercal(\t+\omega)J(\t+\omega) E_1^{\mbox{up}} \\
(\jto Y(\t+\omega))^\intercal J(\t+\omega)E_1^{\mbox{up}}
\end{array} \right ],
\end{equation}
where $E_1^{\mbox{up}}$ is the upper block of $E_1$ at \eqref{e1}. The definition of $S_1(\t)$, see \eqref{S1}, and assumption \eqref{nondegi} easily show that the lower block in  the equation \eqref{calcu5} is identically zero. The upper block of the equation
\eqref{calcu5} is equal to the following term
\begin{align}\label{calcu6}
\phi(\t)-\psi(\t) -X^\intercal_V(\t+\omega)&J(\t+\omega)X_V(\t+\omega)+\\
&-X^\intercal_V(\t+\omega)J(\t+\omega)Z_V(\t+\omega)A(\t), \notag
\end{align}
where 
\[
\phi(\t)=(F_1(\t)X_V(\t))^\intercal\varphi(\t)F_1(\t)\jt Y(\t),
\]
with $\varphi(\t)=J(K(\t+\omega))-J(f(K(\t))$ and
\[
\psi(\t)=[F_1(\t)X_V(\t)-X_V(\t)]^\intercal J(\t+\omega) (F_1(\t)\jt Y(\t)).
\]

Both $\varphi(\t)$ and $F_1(\t)X_V(\t)-X_V(\t)$ are controlled by the error term.
This fact and Lemma \ref{approxlag} show that \eqref{calcu6} is controlled by $\|e_0(\t)\|_\rho$.
Finally the term $(2)$ in the left hand side of \eqref{calcu2} is equal to $\left[ \begin{array}{c}S_2(\t)\\0\end{array} \right ]$ by definition. Since this term is not controlled by the error, we subtract it from  
$M^{-1}(\t+\omega) E(\t)$ to define $B(\t)$, then we get the bound \eqref{estB}. We move $S_2(\t)$ to the coefficients matrix add it to $S_1(\t)$. 

 \end{proof} 
 
 \begin{remark} 
The details to reach  expression \eqref{calcu6} are as follows:
 \begin{align*}
&X_V^\intercal(\t+\omega)J(\t+\omega)J(\t+\omega)[F_1(\t)\jt Y(\t)-\jto Y(\t+\omega)]=\\
&=-\underbrace{[F_(\t)X_V(\t)-X_V(\t+\omega)]^\intercal J(\t+\omega)F_1(\t)\jt Y(\t)}_{\psi}+\\
&\qquad+\underbrace{(F_1(\t)X_V(\t))^\intercal \overbrace{(J(K(\t+\omega)-J(f(K(\t))}^{\varphi}F_1(\t)\jt Y(\t)}_{\phi}+\\
&\qquad\qquad+ \underbrace{(F_1(\t)X_V(\t))^\intercal J(f(K(\t))F_1(\t)\jt Y(\t)}_{(1)}+\\
&\qquad\qquad\qquad-\underbrace{X_V(\t+\omega)J(\t+\omega)\jto Y(\t+\omega)}_{(2)}.
\end{align*}
But we have:
 \begin{align*}
 (1)&=\Omega(F_1(\t)\jt Y(\t), F_1(\t)X_V(\t))= \Omega(\jt Y(\t),X_V(\t))=-I,\\
 (2)&= \Omega(\jto Y(\t+\omega), X_V(\t+\omega))=-I,
\end{align*}
so \eqref{calcu6} follows.
\end{remark}

We will see that the terms $B(\t)\xi(\t)$, $M_e(\t)e_0(\t)$ and $M_e(\t)(\frac{\partial f_\lambda}{\partial\lambda}|_{\lambda=\lambda_0} )\varepsilon_0$ have a quadratic dependence on the error $\|e_0(\t)\|_\rho$, and hence can be controlled. If we omit these terms from \eqref{simplinear} we obtain the linear system:
 \begin{equation}\label{linear}
 \left [ \begin{array}{ccc}
 I_d&S(\t)&0\\
 0&I_d&0\\
 0&A(\t)&I_n
 \end{array}\right ]\xi(\t)-\xi(\t+\omega)=R_0(\t),
 \end{equation}
 where 
 \[
 R_0(\t)=-V^{-1}(\t)Q(\t) e_0(\t)-\Lambda(\t)\varepsilon_0.
 \]
 This linear system can be solved using Proposition \ref{difference}, as we show next:
 
\begin{proposition}\label{solvelinear}
Assume that all hypothesis of Lemma \ref{change} hold. Then there exists a mapping $\xi(\t)$, analytic on $U_{\rho-2\delta}$ and a vector $\varepsilon_0\in\real^{2n}$ such that \eqref{linear} holds for $\xi(\t)$ and $\varepsilon_0$.  Moreover, there exits $c_8$  and $c_9$ depending on $n$, $d$, $\rho$, $r$, $|f_{\lambda_0}|_{C^2,\mathcal{B}}$, $\|DK_0\|_\rho$, $\|N\|_\rho$ , $\left\|\left.\frac{\partial f_\lambda}{ \partial \lambda}\right|_{\lambda=\lambda_0} \right\|_\rho$ such that
 \begin{equation}\label{estxi}
 \|\xi\|_{\rho-2\delta} \leq c_8 \gamma^{-2}\delta^{-2\sigma} \|e_0\|_\rho
 \end{equation}
 \begin{equation}\label{estvarepsilon}
 |\varepsilon_0|\leq c_9|\avg(\Lambda_0)^{-1}|\|e_0\|_\rho
 \end{equation}
 \end{proposition}
 
 \begin{proof}
Since the proof goes through as in the symplectic case, to avoid unnecessary details, we give a short sketch of the proof and refer to \cite{dgjv05} for more details. Let 
 \[R_0(\t)=\left (\begin{array}{c}R_x(\t)\\ R_y(\t) \\ R_z(\t)\end{array} \right ),\quad \xi(\t)=\left ( \begin{array}{c}\xi_x(\t)\\ \xi_y(\t)\\ \xi_z(\t)\end{array} \right ), \]
so \eqref{linear} becomes 
\begin{equation}
   \left\{\begin{array}{l}\label{xyzeq}
  \xi_x(\t)-\xi_x(\t+\omega)=R_x(\t)-S(\t)\xi_y(\t) \\ \\
  \xi_y(\t)-\xi_y(\t+\omega)=R_y(\t)\\ \\
  \xi_z(\t)-\xi_z(\t+\omega)=R_z(\t)-A(\t)\xi_y(\t)
   \end{array}\right.
 \end{equation}
Using the non-degeneracy of the pair $(f_{\lambda},K_\lambda)$ at $\lambda=0$, we can determine $(\varepsilon_0^{d+1},...,\varepsilon_0^{2d})$ in such way that $\avg(R_y)=0$. Then we can apply Proposition \ref{difference} to solve the second equation in \eqref{xyzeq} finding a unique zero average solution $\xi_y(\t)$. After determining $\xi_y(\t)$ one can choose the remaining components of $\varepsilon_0$ so that
\[\avg(R_x-S\xi_y)=\avg(R_z-A\xi_y)=0.\]
Applying again Proposition \ref{difference}, we solve the first and last equation of \eqref{xyzeq} obtaining unique zero average solutions $\xi_x(\t)$ and $\xi_z(\t)$. Proposition \ref{difference} shows that these solutions satisfy the following estimates:
 \begin{align*}
  \|\xi_y\|_{\rho-\delta}&\leq  c^\prime \gamma^{-1}\delta^{-\sigma}\|R_y\|_\rho\\
  \|\xi_x\|_{\rho-2\delta}&\leq c^{\prime\prime} \gamma^{-1}\delta^{-\sigma}\|R_x-S\xi_y\|_{\rho-\delta}\\
   \|\xi_z\|_{\rho-2\delta}&\leq c^{\prime\prime\prime} \gamma^{-1}\delta^{-\sigma}\|R_z-A\xi_y\|_{\rho-\delta}
\end{align*}
The proof of the estimates \eqref{estvarepsilon} and \eqref{estxi}
follow just like in the symplectic case  (see \cite{dgjv05}).
\end{proof}

\begin{cor}\label{estlin} Assume all the hypotheses of the proposition \eqref{solvelinear} hold. then
\begin{align}\label{delta}
&\|\Delta_0\|_{\rho-2\delta}\leq  c \gamma^{-2}\delta^{-2\sigma} \|e_0\|_\rho\\
&\|D\Delta_o\|_{\rho-3\delta}\leq  c \gamma^{-2}\delta^{-(2\sigma+1)} \|e_0\|_\rho\notag.
\end{align}
\begin{equation}\label{approxsollinear}
  \|DG(K_0,\lambda_0)|_{(\Delta_0(\t),\varepsilon_0)}+e_0\|_{\rho-2\delta}\leq c_{12}\gamma^{-3} \delta^{-(3\sigma+1)}\|e_0\|_\rho^2,
\end{equation}
where $\Delta_0(\t)=M^{-1}(\t)\xi(\t)$.
\end{cor}
\begin{proof}
The estimates \eqref{delta} are immediate consequences of the proposition \eqref{solvelinear} and the Cauchy integral formula.
Replacing the solution given by Proposition \ref{solvelinear} into the linearized equation \eqref{DG}  we find that: 
 \begin{align*}
 DG(K_0,&\lambda_0)|_{(\Delta_0(\t),\varepsilon_0)}+e_0(\t)=\\
 &M(\t+\omega)\left(B(\t)\xi(\t)+M_e(\t)e_0(\t)+M_e(\t)\left.\frac{\partial f_\lambda}{\partial\lambda}\right|_{\lambda=\lambda_0}\varepsilon_0\right),
 \end{align*}
 Now \eqref{approxsollinear} follows from \eqref{estMe1} \eqref{estB}  \eqref{estxi} and \eqref{estvarepsilon}.  This establishes 
that indeed, 
we have obtained an approximate solution of the linearized equation \eqref{DG}.
\end{proof}

\begin{remark}
One of the concerns in the  KAM results of the type we are presenting
 here is how many modifying parameter are needed. 
A very lucid discussion regarding this matter can be found in \cite{mo67}.  
A discussion of the dimension of the space of parameters in 
the degenerate cases, can be found in \cite{har12}. 
We note that comparing  \cite{dgjv05}*{Proposition 8} and 
Proposition~\ref{solvelinear}, one sees,  that if the family $f_\lambda$
 consists of exact presymplectic diffeomorphisms, then the dimension of 
parameter space can be reduced by $d$. 
Furthermore, if the initial torus satisfies the
 Kolmogorov \footnote{It is also known as twist condition} 
non-degeneracy condition \cite{kol54}  i.e. if $\avg (S(\t))$ is 
non-singular  where $S(\t)$ is defined in the Lemma  \ref{change}, 
then the dimension of parameter space can be reduced by $d$ again.
The reason is that we can choose  the averages of the tori as parameters. 

In particular, having both families of exact presymplectic 
mappings and Kolmogorov  non-degeneracy condition, it will be enough to consider the parameter space to be $n$ dimensional, see the Vanishing Lemma also.

\end{remark}
 \section{Estimates for the improved step}
\label{sec:improvedstep} 
 In the previous section, we have shown that the linearized equation \eqref{DG} admits approximate solution in all smaller analyticity domains. The estimates blow up if the analyticity loss vanishes. the good point is that they blow up not worse than a power. 

The goal of this section is to show that if $\|\Delta_0\|_{\rho -\delta}$ 
is sufficiently  small,  the new torus $K_1(\t)=K_0(\t)+\Delta_0(\t)$ 
has an error in the invariance equation which is  quadratically
small with respect to the original one (in the smaller  domain). 

 \begin{lemma}\label{step1}
 Assume 
 \[ ( K_0 + \Delta_0)(U_{\rho -\delta},\lambda_0+\varepsilon_0)\subset \mbox{Domain}(f),\]
where $f$ is defined in \eqref{mpara},then 
\begin{equation}\label{eststep}
\|f_{\lambda_0+\varepsilon_0}\circ (K_0+\Delta_0)-(K_0+\Delta_0)\circ T_\omega \|_{\rho-\delta}\leq c \gamma^{-2}\delta^{-4\sigma}\|e_0\|_{\rho},
 \end{equation}
 
where $c$ now involves $\|f\|_{C^2,\mathcal{B}}$ as well as previous quantities. Furthermore, the pair $(f_\lambda, K_1)$  is non-degenerate at $\lambda=\lambda_0+\varepsilon_0$, in the sense of definition \ref{nondeg2}.
 \end{lemma}

Note that the linear equation admits estimates for 
$\Delta$ in any domain $U_{\rho - \delta}$ for 
any $\delta > 0$. If the $\delta$ is very small, the estimates 
blow up. So that if the loss of domain $\delta$ is too small
compared with $\| e_0\|_\rho$.  So that the 
estimates on the step require  some  restrictions on the  of the loss of 
domain $\delta$ allowed. 

Given the estimates on $\Delta, \varepsilon_0$ obtained in 
Corollary~\ref{inver3}, we see that the requirement
on the composition is implied by 
\begin{equation}\label{inductionassumption} 
c \gamma^{-2} \delta^{- (2 \sigma + 1)} \| e_0 \|_\rho  
\le  \eta
\end{equation} 
where $\eta$ is smaller than the distance 
of $K( U_\rho)$ to the complement of the domain of $f$.

\begin{proof}
This is just a simple consequence of the obvious identity obtained by adding and subtracting some terms:
\begin{align*}
&f_{\lambda_0+\varepsilon_0}(K_0+\Delta_0)-(K_0+\Delta_0)\circ T_\omega=\hspace{4cm}\\
&\qquad\underbrace{f_{\lambda_0+\varepsilon_0}(K_0+\Delta_0)-f_{\lambda_0}(K_0)-\left.\frac{\partial f_\lambda}{\partial\lambda}\right|_{\lambda=\lambda_0}(K_0)\varepsilon_0-Df_{\lambda_0}(K_0)\Delta_0}_{(1)}\\
&\qquad+\underbrace{f_{\lambda_0}(K_0)-K_0\circ T_\omega+Df_{\lambda_0}(K_0)\Delta_0-\Delta_0\circ T_\omega+\left.\frac{\partial f_\lambda}{\partial\lambda}\right|_{\lambda=\lambda_0}(K_0)\varepsilon_0}_{(2)}
\end{align*}
The term $(1)$ can be estimated by Taylor theorem, so we have:
\begin{align*}
&\|f_{\lambda_0+\varepsilon_0}(K_0+\Delta_0)-f_{\lambda_0}(K_0)-\left.\frac{\partial f_\lambda}{\partial\lambda}\right|_{\lambda=\lambda_0}(K_0)\varepsilon_0-Df_{\lambda_0}(K_0)\Delta_0\|_{\rho-\delta}\hspace{1cm} \\
&\qquad \leq\frac{1}{2} \|f\|_{C^2,\mathcal{B}}(\| \Delta_o\|_{\rho-\delta}^2+|\varepsilon_0|^2)\leq c \frac{1}{2} \|f\|_{C^2,\mathcal{B}} \gamma^{-2}\delta^{-4\sigma}\|e_0\|_{\rho}
\end{align*}
The term $(2)$ is exactly the left hand side of \eqref{approxsollinear}, so by rearranging the constant we get estimate \eqref{approxsollinear}. Non-degeneracy of the pair $(f_\lambda, K_1)$ comes from the estimates \eqref{delta}, \eqref{estvarepsilon} and the fact that non-degeneracy is an open condition.
\end{proof}

  
\section{Iteration of the Newton method and convergence}
\label{sec:iteration}

We shall now perform our modified Newton method, starting with $f_{\lambda_0}$, $K_0$, $\omega$ and $\rho_0$ satisfying the hypotheses of Theorem \ref{main}, and applying at each step the results of Section \ref{sec:linearized:eq}. We will see that if we choose $\|e_0\|_{\rho_0}$ small enough we will be able to proceed with the iteration so that the equation   
\begin{equation}\label{invariant}
  f_\lambda(K(\t))=K(\t+\omega)
  \end{equation}
has a convergent sequence of approximate solutions
\[ (K_0,\lambda_0), (K_1,\lambda_1), (K_2,\lambda_2), \dots \] 
defined on domains 
\[ U_{\rho_0}\supset U_{\rho_1}\supset U_{\rho_2}\supset\cdots\]
with limit an exact solution $(K_\infty,\lambda_\infty)$, defined on a domain $U_{\rho_\infty}$. 
  
Starting with the approximate solution $(K_0,\lambda_0)$, assume that we have already found the term $(K_m,\lambda_m)$ in this sequence. The next term will take the form:
\[K_m=K_{m-1}+\Delta_{m-1}(\t), \quad \lambda_m=\lambda_{m-1}+\varepsilon_{m-1}\quad (m\geq 1),\]
with $(\Delta_{m-1}(\t),\varepsilon_{m-1})$ an approximate solution of the linear equation
\begin{equation} \label{linearm}
  DG(K_{m-1},\lambda_{m-1})|_{(\Delta_{m-1}(\t),\varepsilon_{m-1})}=-e_{m-1},
\end{equation}
where $e_{m-1}:=G(K_{m-1},\lambda_{m-1})$

The following lemmas are simply restating the lemma \ref{step1} for a general step. 

\begin{lemma}\label{improve} 
Assume that $(K_{m-1},\lambda_{m-1})$ is a non-degenerate (Definition \ref{nondeg2}) approximate solution of \eqref{invariant} such that   
\begin{equation}\label{r}
  r_{m-1} := \|K_{m-1}-K_0\|_{\rho_{m-1}}<r.
\end{equation}
If $\|e_{m-1} \|_{\rho_{m-1}}$ is small enough so that Proposition \ref{solvelinear} applies, then for any $0<\delta_{m-1}<\rho_{m-1}/3$ there exist a function $\Delta_{m-1}(\t)\in\mathcal{P}_{\rho_{m-1}-3\delta_{m-1}}$ and $\varepsilon_{m-1}\in\real^{2d+n}$, such that
\begin{align}
& \| \Delta_{m-1}(\t) \|_{\rho_{m-1}-2\delta_{m-1}} < c_{m-1} \gamma^{-2}
 \delta_{m-1}^{-2\sigma} \|e_{m-1} \|_{\rho_{m-1}}  \notag\\
& \|D \Delta_{m-1}(\t) \|_{\rho_{m-1}-2\delta_{m-1}} < c_{m-1} \gamma^{-2}
 \delta_{m-1}^{-2(\sigma+1)} \|e_{m-1} \|_{\rho_{m-1}}  \label{induct1}\\
& |\varepsilon_{m-1} | \leq c_{m-1}|(\avg(\Lambda_{m-1})^{-1} | \|e_{m-1}\|_{\rho_{m-1}} \notag
 \end{align}
where $c_{m-1}$ is a constant depending on $n$, $d$, $\rho$, $r$, $|f_{\lambda_{m-1}}|_{C^2,\mathcal{B}_r}$, $\|DK_{m-1}\|_\rho$, $\|N_{k-1}\|_\rho$ and $\left\|\left.\frac{\partial f_\lambda}{ \partial \lambda}\right|_{\lambda=\lambda_{m-1}} \right\|_\rho$.

 Moreover if 
\begin{equation} \label{induct2}
r_{m-1}<c_{m-1}\gamma^{-2}\delta^{-2\sigma -1 }_{m-1} \|e_{m-1} \|_{\rho_{m-1}}
\end{equation}
setting 
$K_m=K_{m-1}+\Delta_{m-1}$, $\lambda_m = \lambda_{m-1} + \varepsilon_{m-1}$.
then,
$e_m(\t)=G(K_m,\lambda_m)(\t)$ the error function of
the improved solutions  satisfies
\begin{equation} \label{quadratic} 
\|e_m\|_{\rho_m}\leq c_{m-1} \gamma^{-4} \delta_m^{-4\sigma} \|e_{m-1}\|_{\rho_{m-1}}^2
\end{equation}  
\end{lemma}

\begin{lemma}\label{13} 
Under the same assumptions as in Lemma \ref{improve}, one can improve the constant $c_{m-1}$ such that \eqref{induct2} holds and if 
\begin{equation}\label{induct3}
  c_{m-1} \gamma^{-2} \delta_{m-1}^{-(\sigma+1)}\|e_{m-1}\|_{\rho_{m-1}}\leq \frac{1}{2}
\end{equation}
then
\begin{enumerate}
  \item[(i)] If $(\pi D_1K_{m-1})^\intercal \pi D_1K_{m-1}$ is invertible with inverse $N_{m-1}$, then the matrix $(\pi D_1K_{m})^\intercal \pi D_1K_{m}$ is also invertible with inverse $N_m$ satisfying
  \begin{equation}\label{induct4}
  \|N_m\|_{\rho_m}\leq \|N_{m-1}\|_{\rho_{m-1}} + c_{m-1} \gamma^{-2} \delta_{m-1}^{-(\sigma+1)}\|e_{m-1}\|_{\rho_{m-1}};
  \end{equation}
  \item[(ii)] If $V_{m-1}$ is invertible then $V_M$ is invertible and the inverse satisfies equation \eqref{induct4} with $N$ replaced by $V^{-1}$;
  \item[(iii)] If $\avg(\Lambda_{m-1})$  is invertible then $\avg(\Lambda_m)$ is invertible and the inverse satisfies equation \eqref{induct4} with $N$ replaced by $\avg(\Lambda)^{-1}$.
   \item{(iv)}  The assumption 
\eqref{inductionassumption} ensuring that the 
range of $(K_{m-1} + \Delta_{m-1}, \lambda_{m-1} + \varepsilon_{m-1}) $ 
is inside of the domain of $f$. 
  \end{enumerate}
\end{lemma}

The most important point is that the constants 
$c_m$ depend only on  $n$, $d$, $\rho$, $r$, $|f_{\lambda_{m-1}}|_{C^2,\mathcal{B}_r}$, $\|DK_{m-1}\|_\rho$, $\|N_{k-1}\|_\rho$ and $\left\|\left.\frac{\partial f_\lambda}{ \partial \lambda}\right|_{\lambda=\lambda_{m-1}} \right\|_\rho$.
Hence, when we show that the $K$ does not leave a neighborhood, 
then,  the constants are uniform. 

The convergence of the modified Newton method described above is very standard in KAM theory. Indeed, it has sometimes been formulated as an implicit function theorem. Among the many versions of implicit function theorems, the one of \cite{ze75} is the closest to the problem here. For the sake of completeness, we indicate the main points of the iteration following closely \cite{dela01, dgjv05} and refer to those papers for more details. 
One of the main issues to watch out is that the 
non-degeneracy conditions do not deteriorate much  along 
the iteration and 
that the  assumption 
\eqref{inductionassumption}, which ensures that we 
can define the composition, remains valid.

We start by making the choice of the analyticity loss:
\[\rho_m=\rho_{m-1}-2^{-(m-1)}\delta_0.\]
The most subtle point is to show that the conditions \eqref{r} and \eqref{induct3} are always satisfied. The first one is to guarantee that the new torus always stays in the domain of the $f$ and the second one is to insure the non-degeneracy condition during the iteration.

 The constant $c_{m}$ depends on the quantities $\sigma$, $n$, $d$, $r$, which do not change during the iteration. It also depends on the $\rho_m\leq\rho_0$ and the  following quantities 
 \[ |f_{\lambda_m}|_{C^2,{\mathcal{B}}_r}, \|DK_m\|_{\rho_m}, \|N_m\|_{\rho_m},\|\frac{\partial f_\lambda}{\partial\lambda}\mid_{\lambda=\lambda_m}(K_m)\|_{\rho_m}, |\avg(\Lambda_m)^{-1}|.\] 
 
 This dependence  is polynomial. By similar calculation as follows, can be shown that there exist constant $c$ such that $c_m\leq c$ for $m\geq 0$, see \cite{dgjv05}*{Lemma 13}.  The main point is that we do not get far away from initial torus. Denote $\epsilon_m=\|e_m\|_{\rho_m}$ with the choice of the domain loss, we obtain :
\begin{align}\label{errorstepm}
&\epsilon_m\leq c\gamma^{-4} (2^{-(m-1)}\delta_0)^{-4\sigma} \epsilon_{m-1}^2\leq (c\gamma^{-4})^{(1+2)}(2^{-(m-1)-2(m-2)}\delta_0^{(1+2)})^{-4\sigma}\epsilon_{m-2}^4\\
&\qquad \leq \cdots \leq  (c\gamma^{-4}\delta_0^{-4\sigma})^{1+2+\cdots+2^{m-1}}(2^{4\sigma})^{2^0(m-1)+2(m-2)+\cdots+2^{m-2}}\epsilon_0^{2^m}\notag\\
&\qquad \leq (c\gamma^{-4}\delta_0^{-4\sigma})^{2^m-1}2^{4\sigma(2^m-m)}\epsilon_0^{2^m}\leq (c\gamma^{-4}\delta_0^{-4\sigma}2^{4\sigma}\epsilon_0)^{2^m-1}2^{-4\sigma(m-1)}\epsilon_0\notag,
\end{align}
where we have used that
 \[2^0(m-1)+2(m-2)+\cdots+2^{m-2}=2^{m-1}\sum_{s=1}^{m-1}s2^{-s}\leq2^m-m.\]
One sees that if $\|e_0\|_{\rho_0}$ satisfies the assumption \eqref{cond}, the condition \eqref{induct3} is always satisfied.  It remains to show that \eqref{r} is also satisfied. We denote 
$\kappa=c\gamma^{-4}\delta_0^{-4\sigma}2^{4\sigma}\epsilon_0$. Now, the first estimate in \eqref{induct1}, estimate \eqref{errorstepm} and  the definition of $r_m$ gives us: 
\begin{align}\label{rm}
&r_m\leq r_{m-1}+c_m\gamma^{-2}\delta_{m-1}^{-2\sigma} \|e_{m-1}\|_{\rho_{m-1}} \leq \cdots  \leq c\gamma^{-2}\sigma_0^{-2\sigma}\epsilon_0+c\gamma^{-2}\sum_{j=1}^{m-1}
\delta_j^{-2\sigma}\epsilon_j\\
&\qquad\leq  c\gamma^{-2}\sigma_0^{-2\sigma}\epsilon_0+ c\gamma^{-2}\sigma_0^{-2\sigma}\kappa\epsilon_0\sum_{j=1}^{m-1}2^{2j\sigma}2^{-4\sigma(j-1)}\notag\\
&\qquad  c\gamma^{-2}\sigma_0^{-2\sigma}\epsilon_0\left(1+ \kappa2^{4\sigma}\sum_{j=1}^{\infty}2^{-2j\sigma}\right)= c\gamma^{-2}\sigma_0^{-2\sigma}\epsilon_0\left(1+\kappa\frac{2^{4\sigma}}{2^{2\sigma}-1}\right)\notag.
\end{align}
Again having the assumption \eqref{cond} and calculations \eqref{rm} show that the \eqref{r} is always satisfied.


 \section{Local Uniqueness}
 \label{sec:uniqueness}

The proof of Theorem \ref{uniqtheo} follows exactly the same pattern as the proof of uniqueness for the symplectic case given in \cite{dgjv05}, so we will not reproduce here the same computations. We limit ourselves to some comments and a sketch of the proof, which takes advantage of the fact that in Proposition \ref{difference} two different solutions of \eqref{diffequ} differ by their average. In our situation, one can transfer this difference of averages of two solution to a difference of the phase between them.

Let  $(K_1,\lambda_1)$ and $(K_2,\lambda_2)$ be two solutions as in the statement of Theorem \ref{uniqtheo}. From Taylor's Theorem we have:
\begin{equation} \label{61}
   Df_{\lambda_1}(K(\t))(K_2-K_1) +R(K_1,K_2)=0
\end{equation}
where
\begin{equation}\label{62}
   \|R(K_1,K_2)\|_{\rho}\leq c \|K_2-K_1\|_{\rho}^2 
\end{equation}
Applying the change of variable $M\xi=(K_1-K_2)$, where $M$ is given by \eqref{M}  and replacing $K$ by $K_1$, the linear equation \eqref{61} is transformed to 
\begin{align}
\label{(1)}&\xi_x(\t)-\xi_x(\t+\omega)= (\tilde{R}(K_1,K_2))_x-S_1(\t)\xi_y(\t) \\
\label{(2)}&\xi_y(\t)-\xi_y(\t+\omega)=(\tilde{R}(K_1,K_2))_y\\
\label{(3)}&\xi_z(\t)-\xi_z(\t+\omega)=(\tilde{R}(K_1,K_2))_z-A(\t)\xi_y(\t) 
\end{align}
where
\[ \tilde{R}(K_1,K_2)=-M^{-1}(\t+\omega)R(K_1,K_2). \]
Using Proposition \ref{difference} and \eqref{62}, it follows from \eqref{(2)} that:
\begin{equation}\label{64}
  \| \xi^\bot_y\|_{\rho-2\delta} \leq c\gamma^{-2} \delta^{-2\sigma} \|\tilde{R}(K_1,K_2\|_{\rho}
\end{equation}
  where 
 \[\xi^\bot_y(\t)=\xi_y(\t)-\avg(\xi_y)\]
On the other hand, the average of the right hand sides of \eqref{(1)} and \eqref{(3)} are zero, so the assumption that $\Theta$ (see Theorem \ref{uniqtheo}) has rank $d$ together with the estimates \eqref{62} and \eqref{64} give us
 \[
 \left\|(\xi_x,\xi_y,\xi_z)-(\avg(\xi_x),0,\avg(\xi_z))\right\|_{\rho-2\delta}\leq c\gamma^{-2} \delta^{-2\sigma} \|K_2-K_1\|_{\rho}^2 \]

Similarly to \cite{dgjv05}, this leads to the following lemma.

 \begin{lemma}\label{tau}
There exists a constant $\tilde{c}$ depending on $d$, $n$, $\rho$, $|J|_{\mathcal{B}_r}$, $\|K_1\|_{C^2,\rho}$  such that if 
\[\tilde{c}\|K_2-K_1\|_\rho\leq 1, \]
then one can find $\tau\in\real^{d+n}$ with $|\tau|\leq \|K_2-K_1\|_\rho$ such that
\[ \avg\left( \left[ \begin{array}{c}T_1\\T_3\end{array}\right ] [K_2\circ T_{\tau_1}-K_1] \right)=0,\]
where $T_1$ and $T_3$ are defined by \eqref{M1} after replacing $K$ by $K_1$. Therefore, for any $0<\delta<\rho/2$, we have 
\[\|K_1\circ T_{\tau_1}-K_2\|_{\rho-2\delta}<\hat{c} \gamma^{-2}\delta^{-2\sigma} \|K_1-K_2\|_\rho, \]
for a constant $\hat{c}$ depending on the same parameters as $\tilde{c}$ and  also on $\Theta^{-1}$.
\end{lemma}

We then replace $K_2$ by $K_2\circ T_{\tau_1}$ and repeat the iteration, which is possible since $K_2\circ T_{\tau_1}$ is also an invariant torus and the constants $\tilde{c}$ and $\hat{c}$ do not depend on $K_2$. In this way we produce a convergent sequence of phases $\tau_1, \tau_2,\dots,\tau_m,\dots$ such that the limit $\tau_\infty$ satisfies:
\[\|K_2\circ T_{\tau_\infty}-K_1\|_{\rho/2}=0,\]
therefore completing the proof of Theorem \ref{uniqtheo}.

A different proof which does not require iteration (but requires 
setting a normalization condition) appears in \cite{cadela10}. 

\section{Comparison between Poisson and Presymplectic cases}
\label{sec:comparison}
As we mentioned in the introduction, our results are not applicable to the Poisson dynamical systems. In this section, we only present a simple example to justify our statement about the difference between Poisson and presymplectic diffeomorphisms. Consider the presymplectic structure $\Omega= \d x\wedge \d y$ on $M:=T^*\torus\times \torus$ with standard coordinates  $(x,y,z)$. The corresponding compatible Poisson structure is $\Pi=\frac{\partial}{\partial x}\wedge \frac{\partial}{\partial y}$. For the diffeomorphism $f:M\rightarrow M$ to persevere the    
Poisson bivector $\Pi$, it has to satisfy following condition:
\begin{equation}\label{poisson}
(Df)\left [ \begin{array}{ccc}0&-1&0\\1&0&0\\0&0&0\end{array}\right ] (Df)^\intercal =\left [ \begin{array}{ccc}0&-1&0\\1&0&0\\0&0&0\end{array}\right ],
\end{equation}
which means $f$ will have the form \[f(x,y,z)=(f_1(x,y,z), f_2(x,y,z),f_3(z)),\] and
$\frac{\partial f_1}{\partial x}\frac{\partial f_2}{\partial y}-\frac{\partial f_1}{\partial y}\frac{\partial f_2}{\partial x}=1$. We note that \eqref{poisson} clearly shows the symplectic leafs of  $M$ are invariant under $f$, which is a simple fact from Poisson geometry. For $f$ to be  presymplectic diffeomorphisms of $M$, it has to satisfy
\[ (Df)^\intercal\left [ \begin{array}{ccc}0&-1&0\\1&0&0\\0&0&0\end{array}\right ] (Df) =\left [ \begin{array}{ccc}0&-1&0\\1&0&0\\0&0&0\end{array}\right ] ,\]
this leads $f$ to have the form
\begin{equation}\label{presymplectic}
f(x,y,z)=(f_1(x,y), f_2(x,y),f_3(x,y,z)),
\end{equation}
and again $\frac{\partial f_1}{\partial x}\frac{\partial f_2}{\partial y}-\frac{\partial f_1}{\partial y}\frac{\partial f_2}{\partial x}=1$. In general, there is no canonical way to get a symplectic foliation for a presymplectic manifold. Even if we consider the symplectic foliation arising from Poisson structure, one can see from \eqref{presymplectic} that presymplectic diffeomorphisms, in general, do not preserve symplectic leafs.

\subsection*{Acknowledgements}
The work of R.L has been supported by DMS 0901389 and H. N. A. has been supported by UT Austin-Portugal program, grant number SFRH/BD/40253/2007. We would like to thanks Prof. Rui Loja Fernandes for his patient review of the primary draft and his great help.



\end{document}